\documentclass[11pt]{article}
\usepackage[latin1]{inputenc}
\usepackage{amsmath}
\usepackage{amsfonts}
\usepackage{amssymb}
\usepackage{wasysym}
\usepackage{amsthm}
\usepackage{graphicx}

\newcommand{\R}{\mathbb R}
\newcommand{\eps}{\epsilon}

\newtheorem{theorem}{Theorem}
\newtheorem{lemma}{Lemma}[section]
\newtheorem{proposition}{Proposition}[section]

\newtheorem{remark}{Remark}[section]
\newtheorem{defi}{Definition}[section]
\newtheorem{acknowledgment*}{Acknowledgment}
\newtheorem{assumption}{Assumption}[section]
\newcommand{\be}{\begin{equation}}
\newcommand{\ee}{\end{equation}}

\begin{document}

\setcounter{page}{1}
\setlength{\baselineskip}{1.3\baselineskip}
\thispagestyle{empty}

\Large
\begin{center}{\bf From optimal transportation to optimal teleportation} \end{center}
\normalsize
\begin{center}  G. Wolansky\footnote{gershonw@math.technion.ac.il, Tel +972525236410} \\
Department of Mathematics, Technion, Haifa 32000, israel\end{center}
\begin{center}\today\end{center}

\begin{abstract}
The object of this paper is to study estimates of $\eps^{-q}W_p(\mu+\eps\nu, \mu)$ for small $\eps>0$. Here  $W_p$ is the  Wasserstein metric on positive measures, $p>1$, $\mu$ is a probability measure and $\nu$ a signed, neutral  measure ($\int d\nu=0$). In [W1] we proved uniform (in $\eps$) estimates for $q=1$ provided $\int \phi d\nu$ can be controlled in terms of  $\int|\nabla\phi|^{p/(p-1)}d\mu$, for any smooth function $\phi$.
\par
In this paper we extend the results to the case where such a control fails. This is the case where if, e.g. $\mu$ has a disconnected support, or if the dimension of $\mu$ , $d$ (to be defined) is larger or equal $p/(p-1)$.
\par
 In the latter  case we get   such an estimate  provided $1/p+1/d\not=1$ for  $q=\min(1, 1/p+1/d)$. If   $1/p+1/d=1$ we get a log-Lipschitz estimate.
\par
As an  application we obtain H\"{o}lder estimates in $W_p$  for curves of probability measures which are absolutely continuous in the total variation norm .
\par
In case the support of $\mu$ is disconnected (corresponding to $d=\infty$) we obtain sharp estimates for  $q=1/p$ ("optimal teleportation"):
$$ \lim_{\eps\rightarrow 0}\eps^{-1/p}W_p(\mu, \mu+\eps\nu) = \|\nu\|_{\mu}$$
where $\|\nu\|_{\mu}$ is expressed in terms of optimal transport on a metric graph, determined only by the relative distances between the connected components of the support of $\mu$, and the weights of the measure $\nu$ in each connected component of this support.

\end{abstract}\
\newpage
\section{Introduction}
\subsection{Notation}
\begin{itemize}
\item
$\Omega\subset\R^k$ is  a compact set, equal to the closure of its interior.
\item ${\cal M}:={\cal M}(\Omega)$  is the set of  Borel measures on $\Omega$. ${\cal M}_+$ is the set of non-negative measures in ${\cal M}$. ${\cal M}_1$ is the set of probability (normalized) measures in ${\cal M}_+$.
\item The duality between ${\cal M}(\Omega)$ and $C(\Omega)$ (continuous functions) is denoted by $\langle\mu,\phi\rangle$, where $\mu\in{\cal M}$ and $\phi\in C(\Omega)$. This duality implies an order relation on ${\cal M}$: $\mu_1\geq \mu_2$ iff $\langle\mu_1,\phi\rangle \geq \langle\mu_2,\phi\rangle$ for any  non-negative $\phi\in C(\Omega)$.
\item For $\mu\in{\cal M}_+$, $\text{supp}(\mu)$ is the minimal closed set $A\subset\Omega$ such that $\mu(A)=\mu(\Omega)$.
\item If $\mu\in{\cal M}_+$ then $|\mu|:= <\mu, 1>$ (the "mass" of $\mu$).
\item For $\nu\in{\cal M}$, $\nu_\pm\in {\cal M}_+$ is the  factorization of $\nu$ into positive and negative parts, namely $\nu=\nu_+-\nu_-$   such that $\|\nu\|_{TV}:= |\nu_+|+  |\nu_-|$ is the total variation norm of $\nu$ (in particular, $\nu_\pm$ are mutually singular).
    \item
     ${\cal M}_0$ is  the set of measures $\nu=\nu_+-\nu_-$ where $\nu_\pm\in{\cal M}_+$ and $|\nu_-|=|\nu_+|$. In particular, for any $\nu\in{\cal M}_0$ there exists   a single factorization  $\nu_\pm$.
\end{itemize}
\subsection{Background}
Recall
the definition of the $p-$Wasserstein metric ($p>1$) on  ${\cal M}_1(\Omega)$:
\be\label{wasser} W_p(\mu_1, \mu_2):=\left( \inf_{\pi\in\Pi(\mu_1, \mu_2)}\int_\Omega\int_\Omega |x-y|^p \pi(dxdy)\right)^{1/p}\ee
where
$\mu_0, \mu_1\in {\cal M}_1$,
\be\label{Pidef} \Pi(\mu_1, \mu_2):= \left\{ \pi\in {\cal M}_1(\Omega\times\Omega) \ ; \ \ \pi_{\#1}=\mu_1; \ \ \pi_{\#2}=\mu_2\right\}\ee
Here $\pi_{\#, 1,2}$ represents the first and second marginals of $\pi$ on $\Omega$, respectively.

 The $(C(\Omega))^*$ topology restricted to  ${\cal M}_1$ can be metrized   by  $W_p$ with  $p\geq 1$ ([V1], Theorem 6.9). See also [V], [K], [Ve], [T], [KR],  [Va] among many other sources for this and related metrics.

$W_p$ can be trivially  extended   to any  pair $\mu_1, \mu_2\in{\cal M}_+$ provided $|\mu_1|=|\mu_2|$. This extension is defined naturally by the homogeneity relation
\be\label{recallH} W_p(\alpha\mu_1, \alpha\mu_2)=\alpha^{1/p}W_p(\mu_1, \mu_2)\ee
for $\alpha > 0$.
\par
Note that the total variation  of  $\nu=\nu^+-\nu_-\in {\cal M}_0$ is given by
 $$ \|\nu\|_{TV}= \inf_{\pi\in\Pi(\nu_+, \nu_-)}\int_\Omega\int_\Omega d(x,y) \pi(dxdy)$$
where $d$ is the discrete metric ($d(x,y)=1$ if $x\not= y$, $d(x,x)=0$), see [V1], Theorem 6.15. Since $|x-y|^p<Diam^p(\Omega)d(x,y)$ for any $x,y$ in the compact set $\Omega$, then
\be\label{TV0}W_p(\nu_+, \nu_-)\leq Diam(\Omega)\left\|\nu\right\|^{1/p}_{TV} \ , \ee
hence, by the {\it principle of monotone additivity} (see Proposition \ref{ma} below) and (\ref{recallH}),
\be\label{TV}\eps^{-1/p}W_p(\mu+\eps\nu_+, \mu+\eps\nu_-) \leq Diam(\Omega)\|\nu\|^{1/p}_{TV}\ee
for any $\eps>0$, provided  $\mu\in{\cal M}_+$.
\par
Lemma 5.6 in [W1] (see also Theorem 7.26 in [V]) implies that for any $\nu=\nu_+-\nu_-\in{\cal M}_0$, $\nu_\pm\in{\cal M}_+$ and any probability measure $\mu$
\be\label{lemma1} \liminf_{\eps\searrow 0} \eps^{-1}W_p(\mu+\eps\nu_+, \mu+\eps\nu_-)\geq \sup_{\phi\in {\cal B}_p(\mu)} \langle\nu, \phi\rangle\ee
where, if $p>1$,
\be\label{lemma2} {\cal B}_p(\mu):= \left\{ \phi\in C^1(\Omega); \ \int_\Omega|\nabla\phi|^{p/(p-1)} d\mu \leq 1\right\}\ee
while Lemma 5.7 establishes the opposite inequality  for $\limsup$ in (\ref{lemma1}) (in particular, the existence of a limit), if $\nu$ is absolutely continuous with respect to $\mu$ and both measures are regular enough.
\begin{remark}
Note that for $p=1$ an equality
$$\eps^{-1}W_1(\mu+\eps\nu_+, \mu+\eps\nu_-)= \sup_{\phi\in {\cal B}_1} \langle\nu,\phi\rangle$$
 holds for any $\eps>0$ where ${\cal B}_1$ is the set of $1-$ Lipschitz  functions on $\Omega$.
\end{remark}
In the  cases where there is equality in  (\ref{lemma1}) we obtain
 \be\label{modT}\liminf_{\eps\searrow 0} \eps^{-1}W_p(\mu+\eps\nu_+, \mu+\eps\nu_-)\leq D_p(\mu)\|\nu\|_{TV}\  \ee
 where
 $$ D_p(\mu):=\sup_{\phi\in {\cal B}_p(\mu)}\left(\sup_{x\in \text{supp}(\mu)} \phi(x)-\inf_{x\in \text{supp}(\mu)}\phi(x) \right) \ $$
 is the maximal oscillation  of functions in ${\cal B}_p(\mu)$ restricted to supp($\mu$) and is, of course,  independent of $\nu$. \par
 In this paper we consider the case $D_p(\mu)=\infty$.

 \subsection{Measures of connected support}
  Suppose $\mu$ is a uniform (Lebesgue) measure on a "nice" domain $\Omega\subset\R^d$ (e.g. a ball in $\R^d$). Then ${\cal B}_p(\mu)$ is dense in the unit ball of   the Sobolev space $\mathbb{W}^{1,p\prime}(\Omega)/ \R$ (with respect to that norm) where  $p\prime:=p/(p-1)$. Sobolev  embedding theorem then implies that $D_p(\mu)<\infty$ if $d<p\prime$ (where $\mathbb{W}^{1,p}(\Omega)$ is embedded in $C(\Omega)$), while $D_p(\mu)=\infty$ if $p\prime\leq d$ (see Remark \ref{critsob}).
  \par
   The first result (Theorem \ref{firstT} ) deals with measures $\mu$ of connected support. We introduce the notion of dimensionality
  of measure and define  $d-$connected property of such measures  in  Definitions \ref{dimensia} and \ref{strongdom}.
  \par
   For strong $d-$ connected measure $\mu$ and under the assumption that the support of $\nu$ is contained in the support of $\mu$ we state the existence of a  constant $C$ depending only on $\mu$, and an exponent    $q\in [1/p,1]$ for which
\be\label{modTq}\sup_{\eps> 0} \eps^{-q}W_p(\mu+\eps\nu_+, \mu+\eps\nu_-)\leq C\|\nu\|_{TV}\  \ .  \ee
where
 \be\label{qdef}q=\min(1, 1/d+1/p)  \ \ \text{if} \ \ 1/d+1/p\not= 1  \ .   \ee
The second case $d=p/(p-1)$ (i.e $1/d+1/p=1$) corresponds to the critical Sobolev embedding $\mathbb{W}^{1, p\prime}(\mathbb{R}^d)$
and leads to a  log Lipschitz estimate
\be\label{modTLL}\sup_{\eps>0} \frac{1}{ \eps\ln^{1/p}(1/\eps+1)}W_p(\mu+\eps\nu_+, \mu+\eps\nu_-)\leq C\|\nu\|_{TV}\  \ .  \ee

\subsection{Application:  curves of measures}
 Let $I\subset \R$ be an interval and  $\vec{\mu}\in{\cal M}_+(\Omega\times I)$ such that its $t$ marginal $\mu_{(t)}$ is a probability measure on $\Omega$ for any $t\in I$. Then
  $$ \R\supset I \ni t \mapsto \mu_{(t)}\in {\cal M}_1 \  $$
  can be viewed as a curve in ${\cal M}_1:={\cal M}_1(\Omega)$ parameterized in $I$.
  We say that $\vec{\mu}\in AC^r(I, {\cal M}_1;TV)$ for some $\infty\geq r\geq 1$ if there exists a non-negative  $m\in \mathbb{L}^r(I)$ such that
 $$\| \mu_{(t)}-\mu_{(\tau)}\|_{TV}\leq \int_\tau^t m(s)ds$$
  for any $t>\tau\in I$. \par
  The {\it metric derivative}  [AG1, AGS] of $\vec{\mu}$  with respect to the $TV$ norm is
$$ \vec{\mu}\prime_{(t)}(t):= \lim_{\tau\rightarrow t}\frac{\|\mu_{(t)}-\mu_{(\tau)}\|_{TV}}{|t-\tau|} \ . $$

By    Theorem 1.1.2 in [AGS], the metric derivative exists $t$ a.s. in $I$  for $\vec{\mu}\in AC^r(I,{\cal M}_1;TV)$.
  \par
  On the other hand, ${\cal M}_1$ can also be considered as a metric space with respect to the Wasserstein metric $W_p$.
Recalling  (\ref{TV0}), we observe that, if $\vec{\mu}\in AC^r(I, {\cal M}_1; TV)$ then
$$ W_p(\mu_{(t)}, \mu_{(\tau)})\leq Diam(\Omega)\|\mu_{(t)}-\mu_{(\tau)}\|^{1/p}_{TV}  $$
$$\leq  Diam(\Omega)(\int_\tau^t m)^{1/p}\leq Diam(\Omega)\|m\|_r^{1/p} |t-\tau|^{\frac{r-1}{r}\frac{1}{p}} \ . $$
So, we cannot expect that such a curve $\vec{\mu}\in AC^r(I, {\cal M}_1;TV)$ is  more than \\ $(r-1)/rp-$H\"{o}lder with respect to the Wasserstein metric   $W_p$.

In Theorem \ref{thirdT} we state  that if the support of $\vec{\mu}$ is monotone non-increasing, namely supp($\mu_{(t)}$) $\subseteq$ supp($\mu_{(\tau)}$) for any $t>\tau$,
 and supp($\mu_{(t)}$) is strongly $d-$connected for any $t\in I$   then we can improve this estimate:
 Under the above conditions,    $\vec{\mu}$ is  $q(r-1)/r-$ H\"{o}lder on $I$ in $({\cal M}_1,W_p)$, ($q-$ H\"{o}lder if $r=\infty$) where $q$ given by (\ref{qdef}).

Moreover,  if  $1/p+1/d>1$  then $q=1$ (\ref{qdef})  and  $\vec{\mu}\in AC^r(I, {\cal M}_1;W_p)$ as well. This implies  the existence of a Borel vector field $v\in\mathbb{L}^r\left(I, \mathbb{L}^p(\mu_{(t)})\right)$ for which the continuity equation
  \be\label{ce} \partial_t \mu+\nabla_x\cdot(\mu v)=0 \  \ee
holds as a  distribution [AGS].

To illustrate  the above results, consider
\be\label{mudelta}\mu_{(t)}=m(t)\delta_{x_0}+ (1-m(t))\delta_{x_1}\ee
 where $x_0\not= x_1$ and $t\mapsto m(t)\in(0,1)$ is a non-constant smooth function. Then $\dot{\mu}_{(t)}= \dot{m}(t)(\delta_{x_0}-\delta_{x_1})\in {\cal M}$ and $\|\dot{\mu}_{(t)}\|_{TV}=2|\dot{m}(t)|$.

If we consider the above curve  in $({\cal M}_1,W_p)$ where $p>1$ then  the metric derivative   does not exist. \\ Indeed, since $W^p_p(\mu_{(t)}, \mu_{(\tau)})= |m(t)-m(\tau)|\times |x-x_0|^p$,
all we can obtain is  $1/p$ H\"{o}lder estimate:
$$ \lim_{\tau\rightarrow t}\frac{W_p(\mu_{(t)}, \mu_{(\tau)})}{|t-\tau|^{1/p}} =\lim_{\tau\rightarrow t}\frac{|m(t)-m(\tau)|^{1/p}}{|t-\tau|^{1/p}} |x-x_0|= |\dot{m}|^{1/p}(t) |x-x_0|  \ . $$
Now, replace (\ref{mudelta}) by
\be\label{deltabar} \mu_{(t)}=m(t)\delta_{x_0} + (1-m(t))\delta_{x_1} + \bar{\mu}\ee
(recall (\ref{recallH})) where $\bar{\mu}\in {\cal M}_+$ a stationary (independent of $t$) positive, strongly  $d-$connected measure whose support contains $x_0, x_1$.
Even though $\dot{\mu}=\dot{m}(\delta_{x_0}-\delta_{x_1})$ is the same for both (\ref{mudelta}) and (\ref{deltabar})),   we can find out   that for $\mu$ given by (\ref{deltabar})
$$ \frac{W_p(\mu_{(t)}, \mu_{(\tau)})}{|t-\tau|^{q}} \leq C(\bar{\mu})|\dot{m}|  \  $$
for $q=\min[1, 1/p+1/d]$ (provided $1/p+1/d\not=1$), or the Log-Lipschitz estimate
 $$ \frac{W_p(\mu_{(t)}, \mu_{(\tau)})}{|t-\tau|\ln^{1/p}(1/|t-\tau|)} \leq C(\bar{\mu})  \  $$
 if $1/p+1/d=1$. In particular (\ref{deltabar}) is uniformly  Lipschitz if $1/p+1/d>1$. If this is the case,  it is absolutely continuous in $W_p$. Hence
 the continuity equation (\ref{ce}) is satisfied for some Borel vectorfield $v$ [AGS].

  To elaborate further, let us consider the case
where $\mu_{(t)}$ is supported in an  interval $J\subset\R$ and $[x_0, x_1]\subset  J$:
 $$\mu_{(t)}:= \beta 1_J(dx) +  m(t)\delta_{x_1} + (1-m(t))\delta_{x_0} \  $$
 where $\beta>0$ is a constant.
  Then $\mu_{(\cdot)}$ is strongly $1-$connected (see Definition \ref{strongdom}). Hence  for any $p>1$, $\mu_{(\cdot)}$ is Lipschitz in $W_p$. In particular, it satisfies (\ref{ce}). It can be verified that the transporting vector field is nothing but $$v(x,t)=\beta^{-1}\dot{m}(t) \ \ \text{if} \ \ x_0<x<x_1 ,  \ , \ \ \ v(x_0,t)=v(x_1,t)=0 \ ; \ \ \forall t\in I $$
   and $v$   is arbitrary otherwise .
\par
The case  $q<1$ corresponds, in this context, to a "teleportation": No vector field $v$ exists for which an orbit $\mu_{(t)}$ is transported via the continuity equation (\ref{ce}). In particular, if the support of $\mu$  is disconnected (e.g. $\beta=0$ above).
\subsection{Disconnected support}
In the last  part of the paper  we discuss
 the case of disconnected support of $\mu\in {\cal M}_+$, corresponding to  $d=\infty$. In that case $q=1/p$. Under appropriate condition we state in Theorem \ref{Th2} that there exists   a {\it sharp} limit
 $$ \lim_{\eps\searrow 0} \eps^{-1/p}W_p\left(\mu+\eps\nu_+, \mu+\eps\nu_-\right)= \lim_{\eps\searrow 0} \eps^{-1/p}W_p\left(\mu+\eps\nu, \mu\right):=\|\nu\|_\mu^{1/p}$$
  where $\|\nu\|_\mu$ is defined in terms of an  optimal transport on a {\it finite, metric graph}. This is the rational behind the title "optimal teleportation".
 \par
 To describe the nature of $\|\nu\|_\mu$, consider a finite graph whose vertices are identified with the connected components $A_i$ of the support of $\mu$. The length of an edge connecting two vertices is defined as the $p$ power of the distance between the corresponding supports. We then consider the discrete metric space composed of these vertices, subjected to the geodesic distance corresponding the edge's length defined above.
 \par
 At each vertex $i$ of this graph  let $\bar{\nu}_i\in\R$ be the weight of the measure $\nu$ restricted to corresponding component  $A_i$. By neutrality $\sum_i\bar{\nu}_i=0$.
 \par
 Then   $\|\nu\|_\mu$ is just the optimal transport cost of    $\{\bar{\nu}_i>0\}$ to $\{\bar{\nu}_i<0\}$ for the above defined metric (c.f. Fig 2).

\section{Detailed Description of Main Results}

We start by posing some assumptions on a measure $\mu\in{\cal M}_1$:
\begin{defi}\label{dimensia}
$\mu$ is $d-$connected if {\em supp}($\mu$) is arc-connected and there exists $K,\delta>0$ such that for  any  $x\in$  {\em supp}($\mu$) and  any $0<r<\delta$
\be\label{ar}\mu(B_r(x))\geq Kr^d \ . \ee
\end{defi}
\begin{remark} Condition (\ref{ar}) states, in fact, that $\mu$ is $d-$ {\em Ahlfors regular} from below on its support. See e.g. [J] for more general definitions. Note also that if $\mu$ is $d-$connected then $\mu$ is $d^*-$connected for any $d^*\geq d$.
\end{remark}
Actually, we need a stronger definition for $d-$connected measure:
\begin{defi}\label{strongdom}
$\mu$ is strongly $d-$connected if there exist $ L, K>0$, $2\leq N\in \mathbb{N}$  and a measure space  $(D, \beta)$ such that for any  $x_0,x_1\in$ {\em supp}($\mu$) there  are $k\leq N$ points $y_1=x_0, y_2, \ldots y_k=x_1$ in {\em supp}($\mu$) and  $k-1$   measurable mappings $\Phi_j: J\:= [0,1]\times D\rightarrow$ {\em supp}($\mu$), $j=1, \ldots k-1$    such that
\begin{figure}
 \centering
\includegraphics[height=11.cm, width=14.cm]{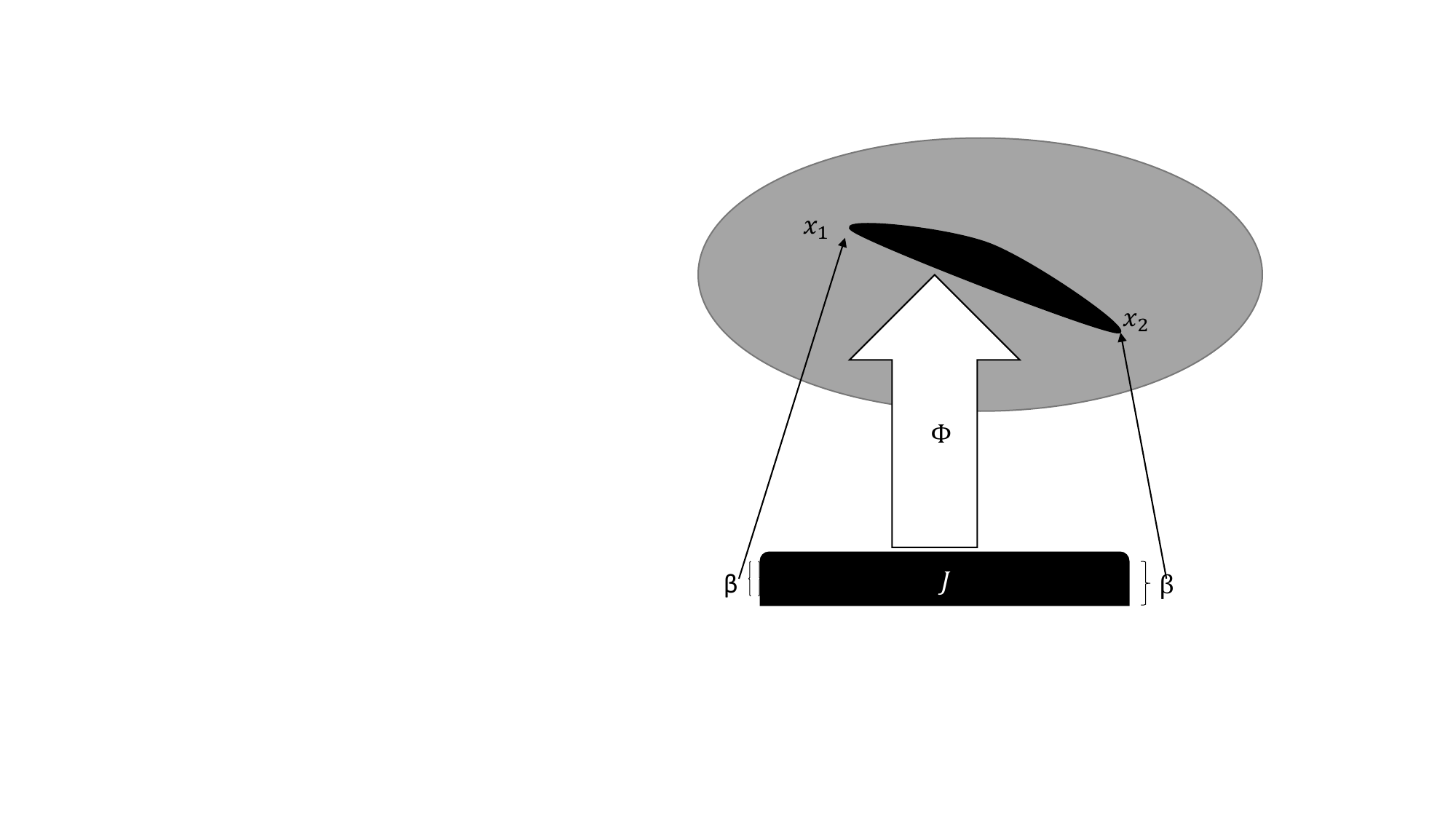}\\
\caption{Mapping of $J$ to $supp(\mu)$ via $\Phi$}\label{1}
\end{figure}
\begin{description}
\item{i)} $\Phi_j(\cdot, b):[0,1]\rightarrow\Omega$ is $L-$Lipschitz  on $[0,1]$ for any $b\in D$.
\item{ii)} $\Phi_j$ is injective  on $(0,1)\times D$,  $\Phi_j(0,b)=y_j$, $\Phi_j(1,b)=y_{j+1}$ for any $b\in D$.
\item{iii)} $\Phi_{j,\#}(\rho)\leq \mu$ where $\rho\in{\cal M}_+(J)$ given by the density \\ $\rho(ds,d\beta)= Ks^{d-1}(1-s)^{d-1}dsd\beta$.
\end{description}
\end{defi}
See Figure 1.
\begin{remark}
We conjecture that $d-$connectedness should be enough for the main results of this paper. Unfortunately we had to adopt the stronger definition for proving these results. Note that
 strong $d-$connected set is also $d-$connected. In fact, {\em supp}($\mu$) is arc connected and $\Phi_1([0,r/L]\times D)\subset B_r(x_0)$  by (i,ii). By (iii), $\mu(B_r(x_0))\geq K\beta(D)\int_0^{r/L} s^{d-1}(1-s)^{d-1}ds$, hence if, say, $r<L/2$ then
$\mu(B_r(x_0))\geq K\beta(D)r^d/(dL^d2^{d})$.
\end{remark}
\par\noindent Examples:
\begin{itemize}
\item Let $Y$ be a convex subset of an $m$ ($\leq k$) dimensional hyperplane in $\R^k$. Let $\mu\geq c{\cal H}^m(T)$ , $c>0$, ${\cal H}^m(Y)$ being the $m-$Hausdorff   measure  on $Y$. Then  $\mu$ is  strongly $m-$connected ($N=2$). The same for a starshaped $Y$ ($N=3$). \\
\item $\mu$ is uniformly distributed on the wedge $$\{ (x,y)\in\R^k\ ; \ 0\leq x\leq 1, \ y\in \R^{k-1}, \ |y|\leq x^\beta \}\  $$
where $k>1$, $\beta\geq 0$. $\mu$ is strongly $\beta (k-1)+1$ connected if $\beta\geq 1$ and strongly $k$ connected if $0\leq\beta\leq 1$ ($N=2$).
\item $\Omega=[0,3]\subset \R$ and $\mu$ has a density proportional to \\ $x\mapsto x(x-1)^2(x-2)^2$. In that case $\mu$ is strongly $3-$connected and $N=3$.
\end{itemize}

\subsection{Connected support}
\begin{theorem}\label{firstT}
Suppose $\mu$ is strongly $d-$connected  ($d\geq 1$) and $\nu=\nu_+-\nu_-\in {\cal M}_0$  such that $\text{supp}(\nu_\pm)\subset \text{supp}(\mu)$.   Then
there exists $C$ depending only on $\mu$  such that
\be\label{firstT1} \sup_{\eps>0} \eps^{-q}W_p\left(\mu+\eps\nu_+, \mu+\eps\nu_-\right)<C\|\nu\|_{TV}\ee
where $q=\min(1, 1/d+1/p)$ provided $p\not= d/(d-1)$.
\par
In the critical case $p=d/(d-1)$ (where $q=1$)
$$\sup_{\eps>0} \frac{1}{\eps\ln^{1/p}(1/\eps+1)}W_p\left(\mu+\eps\nu_+, \mu+\eps\nu_-\right)< C\|\nu\|_{TV} \ . $$

In particular there exists $C=C(\mu)$ for which
\be\label{01}W_p(\mu+\nu_+, \mu+\nu_-) \leq C(\mu)\|\nu\|^q_{TV}\ee
if $p\not=d/(d-1)$,
while if  $p=d/(d-1)$,
\be\label{02}W_p(\mu+\nu_+, \mu+\nu_-) \leq C(\mu)\|\nu\|_{TV}\ln^{1/p}\left(\|\nu\|^{-1}_{TV} +1\right) \  \ee
holds  for any balanced pair $\nu=\nu_+-\nu_-$.
\end{theorem}

\begin{remark}
By Proposition \ref{ma} below we can observe that the {\em optimal} $C(\mu)$ in (\ref{01}, \ref{02}) is monotone non-increasing  in $\mu$, that is $C(\mu_1)\geq C(\mu_2)$ if $\mu_1\leq \mu_2$. By the same Proposition we can also assume that $\nu_\pm$ is a factorization of $\nu$, namely $\|\nu\|_{TV}=|\nu_+|+|\nu_-|$. 
\end{remark}
\begin{remark}\label{critsob}
We may now make a connection between (\ref{lemma1}, \ref{lemma2}), Theorem~\ref{firstT} and the Sobolev embedding Theorem. Consider  the Sobolev space
$$\mathbb{W}^{1,p\prime}(\Omega):= \left\{ \phi\in \mathbb{L}^{p\prime}(\Omega)\ ;  \nabla\phi\in\mathbb{L}^{p\prime} \right\} $$
where $p>1$, $p\prime:= p/(p-1)$ and $\Omega\subset \R^d$ . If $p\prime>d$  then $\mathbb{W}^{1,p\prime}(\Omega)$ is embedded  in the space of bounded continuous functions $C(\Omega)$.  Suppose  $\mu$ is the  Lebesgue measure on a convex set $\Omega\subset \R^d$ (so, in particular, $d-$connected). This implies that the $\mathbb{W}^{1, p^{'}}$ closure of ${\cal B}_p(\mu)$ is embedded in $C(\Omega)$. Let $\nu=\delta_{x_0}-\delta_{x_1}$ where $x_0,x_1\in$ {\em supp}($\mu$).  Then  the right hand side of (\ref{lemma1}) is finite. On the other hand, the case $p\prime>d$ corresponds to the case $q=1$  so (\ref{firstT1}) is consistent with (\ref{lemma1}) in that case.
\par
Recall that the case $p\prime=d$ corresponds to the {\em critical Sobolev embedding} where $\mathbb{W}^{1,d}$  (or ${\cal B}_p(\mu)$) just fails to be embedded  in the space of continuous functions. In that case $D_p(\mu)=\infty$  (see (\ref{modT})). The bound of (\ref{02}) suggests that a Log-Lipschitz estimate corresponds to a critical Sobolev embedding in the context of Wasserstein metric.
\end{remark}

\subsection{Curves of probability measures}
 Let $\vec{\mu}:=\{\mu_{(t)}\}$, $t\in I$  be a curve of probability measures
  $$ \R\supset I \ni t \mapsto \mu_{(t)}\in {\cal M}_1 \ . $$
 Recall  that $\vec{\mu}\in AC^r(I, {\cal M}_1;TV)$ for some $\infty\geq r\geq 1$ if $\exists m\in \mathbb{L}^r(I)$ such that
 $$\|\mu_{(t)}-\mu_{(\tau)}\|_{TV}\leq \int_\tau^t m(s)ds$$
  for any $t>\tau\in I$.
  \begin{theorem}\label{thirdT}
  Suppose $\vec{\mu}\in  AC^r(I,{\cal M}_1;TV)$ for some $\infty\geq r>1$. Assume  also that the support of $\vec{\mu}$ is non-increasing, namely
  {\em supp}($\mu_{(t)}$) $\subseteq$ {\em supp}($\mu_{(\tau)}$) for any $\tau<t\in I$, and {\em supp}($\mu_{(t)}$) is uniformly $d-$connected with respect to $t$ (that is, $N,K,L$ can be chosen independently of $t$ in definition \ref{strongdom}).
  \par
  Then
   \begin{description}
  \item{i)}For any $p>1$,  $p/(p-1)\not= d$,  $\mu$ is uniformly $q(r-1)/r$- H\"{o}lder ($q-$H\"{o}lder if $r=\infty$) in the Wasserstein metric $W_p$ where $q=\min(1, 1/d+1/p)$, namely
  $$ W_p(\mu_{(t)}, \mu_{(\tau)})\leq C |t-\tau|^{q(r-1)/r}$$
  where $C$ is independent of $t\in I$. \par
   If $r=\infty$ and
  $p/(p-1)= d$ then $\mu$ is uniformly log-Lipschitz, that is,
  $$ W_p(\mu_{(t)}, \mu_{(\tau)})\leq C|t-\tau|\left[\ln^{1/p}\left(\frac{1}{|t-\tau|}\right)+1\right]$$
  for some $C$ independent of $t,\tau\in I$.

\item{ii)}
    If  $1<p< d/(d-1)$  then there exists
a Borel vector field $v\in \mathbb{L}^r\left(I, \mathbb{L}^p(\Omega; \mu_{(t)})\right)$ such that the continuity equation
 \be\label{Borelvf}\partial_t\mu+\nabla_x\cdot(v\mu)=0\ee
 is
satisfied in the sense of distributions in $I\times\Omega$.
\end{description}
 \end{theorem}

\section{Proofs for the case of a connected support}
In this section we introduce the proofs of Theorems \ref{firstT}- \ref{thirdT}.
\subsection{Proof of Theorem~\ref{firstT}}
\begin{proposition}\label{fit1}
Suppose $\mu$ is strongly $d-$connected  ($d\geq 1$) and $x_0,x_1\in$ {\em supp}($\mu$).   Then there exists $C=C(\mu)$ depending only on $\mu$ such that
\be\label{firstT1p} \sup_{\eps > 0} \eps^{-q}W_p\left(\mu+\eps\delta_{x_0}, \mu+\eps\delta_{x_1}\right)<C(\mu)\ee
where $q=\min(1, 1/d+1/p)$ provided $p\not= d/(d-1)$.
\par
In the critical case $p=d/(d-1)$ (where $q=1$)
\be\label{firstT2p}\sup_{\eps> 0} \frac{1}{\eps\ln^{1/p}(1/\eps+1)}W_p\left(\mu+\eps\delta_{x_0}, \mu+\eps\delta_{x_1}\right)< C(\mu) \ . \ee
\end{proposition}
\begin{lemma} Proposition \ref{fit1} and Theorem \ref{firstT} are equivalent.
\end{lemma}
\begin{proof}
Obviously Theorem \ref{firstT} implies Proposition \ref{fit1}. To see the opposite direction recall
 (see, e.g. [R])
\be\label{dualp}W^p_p(\mu_1, \mu_2) = \sup_{(\phi,\psi)\in {\cal C}_p(\Omega)} \langle \mu_1,\phi\rangle - \langle \mu_2,\psi\rangle\ee
where
\be\label{Cp} {\cal C}_p(\Omega):= \left\{ (\phi,\psi)\in C(\Omega)\times C(\Omega); \ \phi(x)-\psi(y)\leq |x-y|^p \ \ \ \forall (x,y)\in\Omega\times\Omega\right\}\ee
Without limiting the generality we may assume  $|\nu_+|=|\nu_-|=1$).
Let  $\delta>0$ and $(\bar{\phi}_\delta,\bar{\psi}_\delta)\in {\cal C}_p(\Omega)$ such that 
$$ W^p_p(\mu_1, \mu_2) \leq   \langle \mu_1,\bar{\phi}_\delta\rangle - \langle \mu_2,\bar{\psi}_\delta\rangle+\delta $$ 
  where $\mu_1=\mu+\eps\nu_+$, $\mu_2=\mu+\eps\nu_-$. Let $x_0$ be a maximizer of $\bar{\phi}_\delta$ and $x_1$ a minimizer of $\bar{\psi}_\delta$. Then
\begin{multline}\label{nudelta} W^p_p(\mu+\eps\nu_+, \mu+\eps\nu_-) \leq    \langle  \mu+\eps\nu_+, \bar{\phi}_\delta\rangle - \langle\mu+\eps\nu_-,\bar{\psi}_\delta \rangle+\delta \\
 \leq\langle \mu+\eps\delta_{x_0},\bar{\phi}_\delta\rangle - \langle \mu+\eps\delta_{x_1},\bar{\psi}_\delta\rangle +\delta \leq  W^p_p(\mu+\eps\delta_{x_0}, \mu+\eps\delta_{x_{1}}) +\delta \ . \end{multline}
 Since $\delta>0$ is arbitrary (and independent of $\eps$) we obtain the result. 
\end{proof}
The following result is very easy but useful. For completeness we introduce the proof:
\begin{proposition}\label{ma}
{\em Principle of monotone additivity}: Let $\mu_1, \mu_2, \lambda\in {\cal M}_+$, $|\mu_1|=|\mu_2|$. Then $W_p(\mu_1,\mu_2)\geq W_p(\mu_1+\lambda, \mu_2+\lambda)$.
\end{proposition}
\begin{proof}
Let $\delta>0$. By (\ref{dualp}, \ref{Cp})   there exists $(\phi,\psi)\in {\cal C}_p$ for which
\begin{multline}W^p_p(\mu_1+\lambda, \mu_2+\lambda)\leq
 \langle\mu_1+\lambda, \phi\rangle- \langle\mu_2+\lambda, \psi\rangle+\delta  \\  =\langle\mu_1, \phi\rangle- \langle\mu_2, \psi\rangle + \langle \lambda, \phi-\psi\rangle+\delta\leq\langle\mu_1, \phi\rangle- \langle\mu_2, \psi\rangle +\delta\
\leq W_p^p(\mu_1, \mu_2)+\delta \ . \end{multline} The first inequality follows from $\phi(x)-\psi(x)\leq 0$ for any $x\in\Omega$ by (\ref{Cp}). The third one from (\ref{dualp}). Again, we obtain the desired result since $\delta>0$ is arbitrary. 
\end{proof}
 \subsection{Proof of Proposition~\ref{fit1}}\label{sa}
 To illustrate the proof we start by stating some simplifying assumptions:

   $\Omega$ is one dimensional, e.g $\Omega=[0,1]$, $x_0=0, x_1=1$
  and  \be\label{saa}\mu(ds)=\frac{\rho(s)ds}{\int_0^1\rho(t)dt} \ \ \text{where} \  \rho(s)= K s^{d-1}(1-s)^{d-1} \ . \ee
   For $\mu_1, \mu_2\in {\cal M}_1[0,1]$, let $M_i(s):=\mu_i[0,s]$  be the cumulative distribution function (CDF) of $\mu_i$ for $i=1,2$ respectively.  Let $S^{(i)}$ be the generalized inverses of $M_i$. Then (cf. Theorem 2.18 in [V] for the case $p=2$ and Remark 2.19 there for the general case)
  \be\label{1d} W^p_p(\mu_1, \mu_2)=\int_0^1|S^{(1)}(m)-S^{(2)}(m)|^p dm \ . \ee
  In our case $M_1$ is the CDF of $\mu+\eps\delta_0$ while $M_2$ the CDF of $\mu+\eps\delta_1$. Setting $M=M(s)$ the CDF of $\mu$ and $S=S(m)$ its generalized inverse, then $M_1(s)=M(s)+\eps$ on $(0,1]$ and $M_2(s)=M(s)$ on $s\in[0,1)$, $M_2(1)=1+\eps$.
  The corresponding inverses are
  \begin{description}
  \item{i)}
   $S^{(1)}(m)= 0$ for $m\in[0,\eps]$, $S^{(1)}(m)=S(m-\eps)$ for $\eps\leq m\leq 1+\eps$.
   \item{ii)}   $S^{(2)}(m)=S(m)$ for $m\in[0,1]$ and $S^{(2)}(m)=1$ for
  $m\in[1, 1+\eps]$.
  \end{description}

   Then (\ref{1d}) implies
  $ W^p_p(\mu+\eps\delta_0, \mu+\eps\delta_1)=$
  \be\label{es1}\int_0^\eps|S(m)|^pdm+ \int_\eps^1|S(m)-S(m-\eps)|^p dm+ \int_1^{1+\eps}|S(m-\eps)-1|^pdm \ . \ee
  Since $S$ is monotone non decreasing:
  \be\label{es2}\int_0^\eps|S(m)|^pdm+\int_1^{1+\eps}|S(m-\eps)-1|^pdm\leq \eps\left[ S^p(\eps)  +|1-S(1-\eps)|^p\right]\ee
while
\be\label{es3}  \int_\eps^1|S(m)-S(m-\eps)|^p dm= \eps^p\int_\eps^{1-\eps}\left|\frac{dS}{dm}\right|^p dm\left(1+o(1)\right)\ . \ee
     By the simplifying assumptions  (\ref{saa})
      $$ \kappa_1s^d\leq M(s) \leq \kappa_2s^d \ \ , \ \ \kappa_1(1-s)^d\leq 1- M(s) \leq \kappa_2(1-s)^d$$ for some $0<\kappa_1<\kappa_2$ where $s\in [0,1]$.  Hence
  $$  \kappa_2^{-1/d}m^{1/d}\leq S(m)\leq   \kappa_1^{-1/d}m^{1/d} \ \  , \ \ \kappa_2^{-1/d}(1-m)^{1/d}\leq 1-S(m)\leq   \kappa_1^{-1/d}(1-m)^{1/d}
  $$
for $m\in [0, 1]$.
 From this and  $S^{'}(m):= dS/dm= 1/\rho(S(m))$
 $$ S^{'}(m)= \frac{1}{\rho(S(m))}\leq \kappa \min\{ m^{1/d-1}, (1-m)^{1/d-1}\} $$
 for some $\kappa>0$ and $m\in[0,1]$.
   It follows from (\ref{1d}-\ref{es3}) that
  \begin{description}
  \item{i)} If $p<d/(d-1)$ then
  $W^p_p(\mu+\eps\delta_0, \mu+\eps\delta_1)\leq O(\eps^p)$.
 \item{ii)}  if $p=d/(d-1)$ then
   $W^p_p(\mu+\eps\delta_0, \mu+\eps\delta_1)\leq O\left(\eps^p\ln(1/\eps+1)\right)$.
  \item{iii)}  if $p>d/(d-1)$ then
    $W^p_p(\mu+\eps\delta_0, \mu+\eps\delta_1)\leq O\left(\eps^{p/d+1}\right)$.
\end{description}

In the general case, we provide the estimate (i-iii) for \\ $W^p_p(\mu+\eps\delta_{y_j}, \mu+\eps\delta_{y_{j+1}})$ for $j=1, \ldots k-1$ (see Definition
~\ref{strongdom}). Indeed, since $W_p$ is a metric we get by the triangle inequality
$$ W^p_p(\mu+\eps\delta_{x_0}, \mu+\eps\delta_{x_1})\leq \left(\sum_{j=1}^kW_p(\mu+\eps\delta_{y_j}, \mu+\eps\delta_{y_{j+1}})\right)^p \ . $$

Consider $J,\Phi_j$ as in Definition~\ref{strongdom}. We my replace $\mu$ by the measure $\hat{\mu}:=K\Phi_{j,\#} \rho$. Indeed, by assumption, $\hat{\mu}\leq \mu$ and the inequality
\be\label{settT} W_p(\mu+\eps\delta_{y_j}, \mu+\eps\delta_{y_{j+1}}) \leq W_p(\hat{\mu}+\eps\delta_{y_j}, \hat{\mu}+\eps\delta_{y_{j+1}})\ee
is evident by monotone additivity (Proposition \ref{ma}).
\par
Let $(X_\eps,\sigma)$ be a reference measure space  such that $\int_{X_\eps} d\sigma=\eps+\int_\Omega d\hat{\mu}$. If $T^{(i)}:X_\eps\rightarrow\Omega$, $i=1,2$, is a pair of Borel mappings such that  $T^{(1)}_\#\sigma=\hat{\mu}+\eps\delta_{y_j}$,  $T^{(2)}_\#\sigma=\hat{\mu}+\eps\delta_{y_{j+1}}$ then
\be\label{setT}W^p_p(\hat{\mu}+\eps\delta_{y_j}, \hat{\mu}+\eps\delta_{y_{j+1}})\leq\int_{X_\eps}\left|T^{(1)}(x)-T^{(2)}(x)\right|^p\sigma(dx) \ . \ee
We now construct $(X_\eps,\sigma)$ as follows:

Let  $M=M(s)$ be the CDF of $\rho$ (c.f. (\ref{saa})).  Set $\bar{M}:= M(1)$.  Then
$$ X_\eps:= \left\{ (m, \beta)\in [0,\bar{M}+\eps]\times D\right\}$$
and $\sigma$ is a multiple  $dmd\beta$  on $X_\eps$, normalized according to $\int_{X_\eps} d\sigma=\int_\Omega d\hat{\mu}+\eps$.

Let $S=S(m)$ the  generalized inverse of $M=M(s)$, and extend it to $X_\eps$ by $S(m,\beta)=S(m)$.
In analogy with one-dimensional case above, set
\begin{description}
  \item{i)}
   $S^{(1)}(m,\beta)= 0$ for $m\in[0,\eps]$, $S^{(1)}(m,\beta)=S(m-\eps,\beta)$ for $\eps\leq m\leq \bar{M}+\eps$.
   \item{ii)}   $S^{(2)}(m,\beta)=S(m,\beta)$ for $m\in[0,\bar{M}]$ and $S^{(2)}(m,\beta)=\bar{M}$ for
  $m\in[\bar{M}, \bar{M}+\eps]$.
  \end{description}
  By construction, $S^{(i)}:X_\eps\rightarrow J$ satisfy $$S^{(1)}_\#\sigma= \rho ds +\eps\delta_{s=0} d\beta \ ; \ \ \ \ S^{(2)}_\#\sigma= \rho ds+\eps\delta_{s=1}d\beta \ . $$
   From Definition~\ref{strongdom} it follows that $T^{(1,2)}:= \Phi_j\circ S^{(1,2)}$ satisfy $T^{(1)}_\# \sigma=\tilde{\mu}+\eps\delta_{y_{j}}$, $T^{(2)}_\# \sigma=\tilde{\mu}+\eps\delta_{y_{j+1}}$. Then Definition~\ref{strongdom}-(i')  yields
  $$\int_{X_\eps}\left|T^{(1)}(m,\beta)-T^{(2)}(m,\beta)\right|^pdmd\beta\leq L^p\int_{X_\eps}\left|S^{(1)}(m,\beta)-S^{(2)}(m,\beta)\right|^p dmd\beta \ . $$
  We now proceed as in the one-dimensional case to obtain the proof  by (\ref{settT}, \ref{setT}) via (\ref{es1}-\ref{es3}), in the general case.
\subsection{Proof of Theorem \ref{thirdT}}
\begin{proposition}\label{secondT}
Suppose  $\mu\in{\cal M}_+$, $\nu\in {\cal M}_0$ and $\mu+\nu\in{\cal M}_+$.  Under the assumptions of Theorem \ref{firstT}, there exists $\bar{C}=\bar{C}(\mu)$ for which
\be\label{firstT10}W_p\left(\mu+\nu, \mu\right)<\bar{C}\|\nu\|^q_{TV}\ee
where $q=\min(1, 1/d+1/p)$ provided $p\not= d/(d-1)$.
\par
In the critical case $p=d/(d-1)$ (where $q=1$)
$$W_p\left(\mu+\nu, \mu\right)< \bar{C}\|\nu\|_{TV}\ln\left(\|\nu\|^{-1}_{TV}+1\right) \ . $$
\end{proposition}
For the proof of this proposition we need the following auxiliary lemma
\begin{lemma}\label{fit}
Suppose $\mu,\nu_-\in{\cal M}_+$, $\mu$ is   $d-$connected and   $\nu_-\leq \mu$. Then there exists $\tilde{\nu}\in{\cal M}_+$ such that $\tilde{\nu}-\nu_-\in{\cal M}_0$,
$\tilde{\nu}\leq \mu/2$, $\tilde{\nu}+\nu_-\leq \mu$ and  a constant $\hat{C}(\mu)$  such that $$W_p(\mu-\nu_-, \mu-\tilde{\nu})< \hat{C}(\mu)|\nu_-|^q $$
with $q=\min\{1, 1/p+1/d\}$.
\end{lemma}
\begin{proof}
Given $\eps_0>0$ it is enough to prove it for any $|\nu_-|<\eps_0$. So, let $|\nu_-|=\eps<\eps_0$. Let $\beta>0$ large enough (independent of $\eps$). For any such $\eps$ we  divide the domain $\text{supp}(\mu)$ into essentially disjoint, measurable cells $U_i\subset\Omega$
 such that\\  $\cup U_i\supset$ supp($\mu$), $U_i\cap U_j=\emptyset$ where $i\not= j$, and such that
\begin{description}
\item{i)} Each cell contains a ball of radius $r_\eps:=(4/K)^{1/d}\eps^{1/d}$ whose center is in $supp(\mu)$.  Here $K$ is given by Definition \ref{strongdom}.
    \item{ii)} Each cell is contained in a concentric ball of radius $\beta r_\eps$.
\end{description}
The existence of such a division can easily be demonstrated by tilling a neighborhood of $supp(\mu)$ by, say, identical boxes. The constant $\beta$ depends  {\it only}  on the dimension of the embedding domain.

Let $\nu_i$ be the restriction of $\nu_-$ to $U_i$, $\alpha_i:= |\nu_i|$, the mass of $\nu_-$ contained in $U_i$.
By assumption, $\sum_i\alpha_i=\eps$.
\par
By $d-$connectedness (see Definition~\ref{dimensia}) and (i), $\mu(U_i)\geq 4\eps$ for any $i$.  Let
$$ V_i:=\{ x\in U_i; \ d\nu_i/d\mu\leq 1/2 \} \  $$
where $d\nu_i/d\mu$ stands for the Radon-Nikodym derivative. (Note that $d\nu_i/d\mu\leq 1$ since $\nu_i\leq\nu_-\leq \mu$). Then
$$ \eps\geq\alpha_i\geq \int_{U_i-V_i}(d\nu_i/d\mu)d\mu \geq \frac{1}{2}\mu(U_i-V_i)$$
hence $\mu(U_i-V_i)\leq 2\eps$, so $\mu(V_i)\geq 4\eps-2\eps\geq 2\alpha_i$.

Let $\tilde{V}_i\subset V_i$, a measurable set such that $\mu(\tilde{V}_i)=2\alpha_i$. Define $\tilde{\nu}_i$ as the restriction of $\mu/2$ to $\tilde{V}_i$. In particular, $|\tilde{\nu}_i|=\alpha_i$, and $\tilde{\nu}_i\leq \mu/2$.

Let now $\tilde{\nu}:=\sum_i\tilde{\nu}_i$. Since the sets $\tilde{V}_i$ are mutually disjoint, $\tilde{\nu}\leq \mu/2$, i.e. $d\tilde{\nu}/d\mu\leq 1/2$ $\mu-$a.e. Moreover, $d\tilde{\nu}/d\mu+ d\nu_-/d\mu\leq 1$  $\mu$-a.e, since $d\tilde{\nu}/d\mu=0$ if $d\nu_-/d\mu> 1/2$ by construction while $d\nu_-/d\mu\leq 1$ by the assumption $\nu_-\leq \mu$.  So $\nu_-+\tilde{\nu}\leq \mu$. Finally,  $|\tilde{\nu}|=|\nu_-|=\eps$, so $\tilde{\nu}-\nu_-\in{\cal M}_0$.

Since the diameter of the set $U_i$ is not larger than $2\beta r_\eps$ (c.f. (ii)), the $W_p^p$ cost for shifting a mass $\alpha_i$  within  $U_i$     is not larger that $\alpha_i(2\beta r_\eps)^p$.  Hence
\be\label{sofsof0}W^p_p\left(\tilde{\nu}_i, \nu_i\right)\leq \alpha_i(2\beta r_\eps)^p=\alpha_i(2\beta)^p\left(\frac{4}{K}\right)^{p/d}\eps^{p/d}  \  \ee
Recalling $\nu_-=\sum \nu_i$, $\tilde{\nu}=\sum\tilde{\nu}_i$ we get $W_p^p(\tilde{\nu}, \nu_-)\leq \sum_i W^p_p\left(\tilde{\nu}_i, \nu_i\right)\leq $
 $$  \sum_i (2\beta)^p\left(\frac{4}{K}\right)^{p/d}\alpha_i\eps^{p/d}= (2\beta)^p\left(\frac{4}{K}\right)^{p/d}\eps^{p/d+1}= (2\beta)^p\left(\frac{4}{K}\right)^{p/d}|\nu_-|^{p/d+1}  $$
were we used $\sum\alpha_i=\eps=|\nu_-|$.

Let now $\lambda:= \mu-\nu_--\tilde{\nu}\geq 0$. Then
$$W_p(\mu-\nu_-, \mu-\tilde{\nu})=W_p(\lambda+\tilde{\nu}, \lambda+\nu_-)\leq W_p(\tilde{\nu}, \nu_-)\leq  (2\beta)\left(\frac{4}{K}\right)^{1/d}|\nu_-|^{q} \  $$
by Proposition \ref{ma}.  \end{proof}

{\it Proof  of Proposition \ref{secondT}}:
Let $\nu=\nu_+-\nu_-$. We may assume by the principle of  monotone additivity that $\nu_\pm$ are the positive/negative parts of $\nu$, i.e.
$\|\nu\|_{TV}=|\nu_+|+|\nu_-|$. Let  $\bar{\mu}:= \mu+\nu_+$. Then, by the triangle inequality,
\begin{multline}\label{000}W_p(\mu+\nu, \mu)=W_p(\bar{\mu}-\nu_-, \bar{\mu}-\nu_+)
\leq W_p(\bar{\mu}-\nu_-, \bar{\mu}-\tilde{\nu})
+W_p(\bar{\mu}-\tilde{\nu}, \bar{\mu}-\nu_+)\end{multline}
 where $\tilde{\nu}$ is as in Lemma~\ref{fit} (in particular, $\bar{\mu}$ majorizes  $\tilde{\nu}$, as well as $\nu_-, \nu_+$).
Since $\bar{\mu}\geq \mu$ we get by  monotone additivity and Lemma~\ref{fit}
\be\label{001}W_p(\bar{\mu}-\nu_-, \bar{\mu}-\tilde{\nu})\leq W_p(\mu-\nu_-, \mu-\tilde{\nu})\leq \hat{C}(\mu)|\nu_-|^q \equiv 2^{-q}\hat{C}(\mu)\|\nu\|^q_{TV}\ . \ee
Setting $\tilde{\mu} = \mu-\tilde{\nu}:= \bar{\mu}-\nu_+-\tilde{\nu}$ we get
\be\label{predelta} W_p(\bar{\mu}-\tilde{\nu}, \bar{\mu}-\nu_+)= W_p(\tilde{\mu}+\nu_+, \tilde{\mu}+\tilde{\nu}) \ . \ee
Now,  $\tilde{\mu}\geq\mu/2$ by Lemma \ref{fit}, and since $\|\nu_+-\tilde{\nu}\|_{TV}\leq |\nu_+|+|\tilde{\nu}| = |\nu_+|+|\nu_-|= \|\nu\|_{TV}$,
we obtain from Theorem \ref{firstT}, (\ref{predelta})   and Proposition {\ref{ma}
\be\label{do2} W_p(\bar{\mu}-\tilde{\nu}, \bar{\mu}-\nu_+) \leq C(\mu/2)\left\{ \begin{array}{c}
                                                                                      \|\nu\|^q_{TV} \ \text{if}\  \ p\not=\frac{d}{d-1} \\
                                                                                      \|\nu\|_{TV}\left(\ln(\|\nu\|^{-1}_{TV}+1\right) \ \ \text{if}\  \ p=\frac{d}{d-1}
                                                                                    \end{array}\right\}
  \ . \ee
The proposition now follows from (\ref{000},\ref{001},\ref{do2}) where $\bar{C}(\mu)=2^{-q}\hat{C}(\mu)+C(\mu/2)$.
\begin{proof}
{\em  of Theorem \ref{thirdT}:} \\
i) \ \
Given $t>\tau\in I$, let $\nu=\mu_{(t)}-\mu_{(\tau)}$. Note that { supp}($\nu$)$\subseteq${supp}($\mu_{(\tau)}$). Since $\vec{\mu}\in AC^r(I, TV)$
 \be\label{hara}\|\nu\|_{TV} \leq \int_\tau^t m \leq \|m\|_r|t-\tau|^{1-1/r} \ . \ee
  The assumptions of Theorem \ref{firstT} are satisfied so we obtain the result by   Proposition \ref{secondT}, upon estimating  $\|\nu\|_{TV}$ by (\ref{hara}).
\end{proof}

ii) \ \ Since  $1<p< d/(d-1)$   we get $q=1$ so (\ref{firstT10}), where $\mu:=\mu_{(\tau)}, \nu=\mu_{(t)}-\mu_{(\tau)}$,  with  the first inequality in  (\ref{hara}) imply
$$W_p(\mu_{(t)}, \mu_{(\tau)}) \leq C\int_\tau^t m \ . $$
Since  $\mu_{(\cdot)}\in AC^r(I, TV)$ by the assumption, then $m\in \mathbb{L}^r$ so $\vec{\mu}\in AC^r(I, W_p)$ as well.  The existence of a  vector field satisfying (\ref{Borelvf}) follows from   Theorem 8.3.1 in [AGS] (see also Theorem 5 in [L]).
\section{Optimal teleportation and disconnected support}
In the case of disconnected support of $\mu$ we obtain the following result:
\begin{assumption} \label{ass1}  .
\begin{enumerate}
\item $\mu\in {\cal M}_1$  and $\text{supp}(\mu)$ is composed of a finite number $(m\geq 2)$ of disjoint components $\mu =
\sum_{j=1}^m\mu_j$ where {\em supp}($\mu_i$) $\cap$ {\em supp}($\mu_j$) = $\emptyset$ for any $ i \not= j$.
\item  Each $\mu_i$ satisfies the assumptions of Theorem~\ref{firstT}.
\item  $\nu=\nu_+-\nu_-\in {\cal M}_0$, {\em supp}($\nu_+$)$\cup${\em supp}($\nu_-$) $\subset$ {\em supp}($\mu$).
\end{enumerate}
\end{assumption}
\begin{defi}\label{def1}
 $A_i$ := {\em supp}($\mu_i$) are the connected  components of $\text{supp}(\mu)$.
\begin{description}
\item{i)} $\bar{\nu}_j:= \langle\nu, 1_{A_j}\rangle$.   By Assumption~\ref{ass1}-(3),  $\sum_{j=1}^m\bar{\nu}_j=0$.
\item{ii)}  $V:= \{1...,m\}$, $V_+ := \{j \in V ; \bar{\nu}_j > 0\}$,  $V_- := \{j \in V;  \bar{\nu}_j < 0\}$.
 \item{iii)} For $i,j \in V$, $|E|_{i,j} := dist^p(A_i,A_j)\equiv
 \min_{x\in A_i, y\in A_j}|x-y|^p$.
 \item{iv)}
$G:=(V,E)$ is a complete graph (i.e. any two vertices are connected by an edge) whose vertices $V$ and  the length
of the edge $E_{i,j}$ connecting $i$ to $j$ is $|E|_{i,j}$.
\par
  \item{v)} Let  ${\cal O}_{i,j}$ is the set of all orbits in $V$ connecting $i$ to $j$, that is, $o_{i,j}\in {\cal O}_{i,j}$ if
 $$  o_{i,j} =\{o_{i,j}^{(1)}, \ldots o_{i,j}^{(n)}\} \subset V $$
 such that $o_{i,j}^{(1)}=i$, $o_{i,j}^{(n)}=j$. The length of such an orbit is $|o_{i,j}|=n$ in that case.

   Given $i, j \in V$, $d(i,j)$ is   the geodesic distance corresponding to $(V,E)$. That is:
\be\label{oo} d(i,j) := \min_{o_{i,j}\in {\cal O}_{i,j}}\sum_{l=1}^{|o_{i,j}|-1} |E|_{o_{i,j}^{(l)}, o_{i,j}^{(l+1)}}\ee

 See Figure 2 for an illustration.
\item{vi)}
Let now $\bar{\nu}_i>0$ be the charge associated with the vertex $i\in V_+$ , and  $-\bar{\nu}_j > 0$ the charge associated with $j\in V_-$. Let $\|\nu\|_\mu$ be the optimal cost of transportation of
$\sum_{i\in V_+} \bar{\nu}_i\delta_i$  to
$\sum_{j\in V_-} (-\bar{\nu}_j)\delta_j$
 subjected to the graph metric $d(i,j)$. That is:
 \be\label{mmk}\|\nu\|_\mu:=\min_{\lambda\in\lambda(\nu)}\sum_{i\in V_+}\sum_{j\in V_-}\lambda_{i,j}d(i,j):=\sum_{i\in V_+}\sum_{j\in V_-}\lambda^*_{i,j}|d(i,j) \ee
where $\lambda(\nu)$ is the set of non-negative $|V_+|\times |V_-|$ matrices $\{\lambda_{i,j}\}$ which satisfy:
\\
$\sum_{j\in V_-}\lambda_{i,j}=\bar{\nu}_i \ \ \text{if} \ i\in V_+$
\\
$\sum_{i\in V_+}\lambda_{i,j}=-\bar{\nu}_j \ \ \text{if} \ j\in V_-$ .
\end{description}

\end{defi}
 \begin{figure}
 \centering
\includegraphics[height=13.cm, width=20.cm]{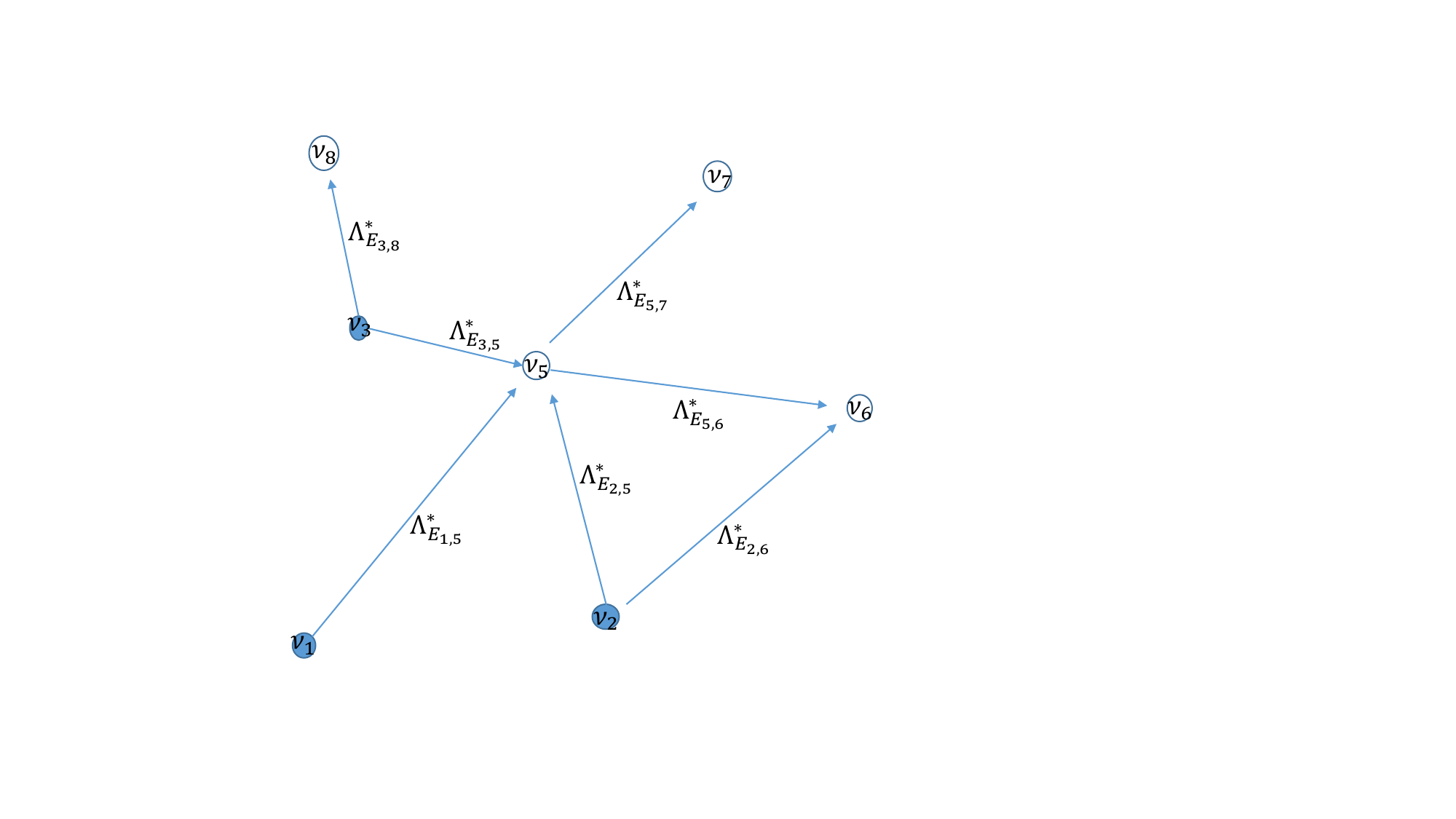}\\
  \caption{Transfer plan via a directed graph. Sources ($\bar{\nu}_i>0$) are filled circles while sinks ($\bar{\nu}_i<0$) are empty circles.
   (c.f. Definition \ref{def1}-(i)). Geodesic arcs: $(1\mapsto 5)= (1,5)$;
   \ $(1\mapsto 6)= (1,5,6)$, \ $(1\mapsto 7)= (1,5,7)$,\ $(1\mapsto 8)= (1,5,3,8)$, \ \ $(2\mapsto 5)= (2,5)$, \ $(2\mapsto 6)= (2,6)$, \ $(2\mapsto 7)= (2,5,7)$, \ $(2\mapsto 8)= (2,5,3,8)$, \ \
   $(3\mapsto 5)= (3,5)$,
   \ $(3\mapsto 6)= (3,5,6)$, \ $(3\mapsto 7)= (3,5,7)$, \ $(3\mapsto 8)= (3,8)$ ; \ \ \ \ \  Weighed arcs:
  $\Lambda^*_{E_{1,5}}=\lambda^*_{1,5}+\lambda^*_{1,6}+\lambda^*_{1,7}$, \ $\lambda^*_{E_{2,5}}=\lambda^*_{2,5}+\lambda^*_{2,7}$,
   \ $\Lambda^*_{E_{2,6}}=\lambda^*_{2,6}$, \
  $\Lambda^*_{E_{5,6}}=\lambda^*_{3,6}+\lambda^*_{1,6}$,
  $\Lambda^*_{E_{5,7}}=\lambda^*_{1,7}+\lambda^*_{2,7}+\lambda^*_{3,7}$ , \
  $\Lambda^*_{E_{3,5}}=\lambda^*_{3,5}+\lambda^*_{3,6}+\lambda^*_{3,7}$. \ \ It is assumed that $\bar{\nu}_3$ is large enough to supply $\bar{\nu}_8$, so $\lambda^*_{1,8}=\lambda^*_{2,8}=0$. Otherwise, the arrow $E_{3,5}$ should be reversed, and $\Lambda^*_{E_{1,5}}=\lambda^*_{1,5}+\lambda^*_{1,6}+\lambda^*_{1,7}+\lambda^*_{1,8}$, $\lambda^*_{E_{2,5}}=\lambda^*_{2,5}+\lambda^*_{2,7}+\lambda^*_{2,8}$, and $\Lambda^*_{E_{5,3}}=\lambda^*_{1,8}+\lambda^*_{2,8}$.  }
  \label{2}
\end{figure}
\begin{theorem}\label{Th2}
If $\infty> p > 1$  and $\mu, \nu:=\nu_+-\nu_-$ satisfy Assumption~\ref{ass1} then
$$ \lim_{\eps \searrow 0} \eps^{-1/p} W_p(\mu, \mu+\eps\nu)=\|\nu\|_\mu^{1/p}$$
\end{theorem}
\subsection{Proof of Theorem~\ref{Th2}}
We first state the inequality
$$\liminf_{\eps\searrow 0}\eps^{-1/p}W_p(\mu,\mu+\eps\nu)\geq\|\nu\|_\mu^{1/p} \ . $$
From the  principle of monotone additivity it is enough to prove
\be\label{ineq1}\liminf_{\eps\searrow 0}\eps^{-1/p}W_p(\mu+\eps\nu_+,\mu+\eps\nu_-)\geq\|\nu\|_\mu^{1/p} \ . \ee
Recall the dual formulation (\ref{dualp}, \ref{Cp}).
In fact, it is enough to restrict to $(\phi,\psi)\in{\cal C}_p(\text{supp} (\mu))\equiv {\cal C}_p(\cup A_i)$.  In the special case $\psi(x)=\phi(x):= z_i$ is a constant   over  $A_i$ we get
\be\label{2.2}W_p^p(\mu+\eps\nu_+, \mu+\eps\nu_-)\geq \eps\sum_{i\in\bar{V}} z_i\bar{\nu}_i\ee
provided $z_i-z_j\leq |x-y|^p$
 for any $x\in A_i, y\in A_j$.
 In particular, if
  $z_i-z_j\leq d(i,j)$
(see definition~\ref{def1}-(iii, v)).
  From (\ref{2.2}) and Definition~\ref{def1}-(ii) we get
\be\label{2.3} W_p^p(\mu+\eps\nu_+, \mu+\eps\nu_-)\geq \eps\sup_{\{z\}}\sum_{i\in\bar{V}} z_i\bar{\nu}_i\ee
where the supremum is on all possible values of $\{z_1\ \ldots , z_{\#\bar{V}}\}$ which satisfy $z_i-z_j \leq d(i,j)$ for any $i,j\in\bar{V}$. Since $d(\cdot, \cdot)$ is a metric on the graph $(V,E)$  via Definition~\ref{def1}-(v)  we recall the dual formulation of the metric Monge problem, or the so called Kantorovich Rubinstein Theorem (Theorem  1.14 in [V] or [R1]) in discrete version:
\be\label{2.4}\|\nu\|_\mu = \sup_{\{z\}}\sum_{i\in V} z_i\bar{\nu}_i \ \ ; \ \  z_i-z_j  \leq d(i,j) \  \ee
(see also Definition~\ref{def1}-(vi)).
Then (\ref{ineq1}) follows from (\ref{2.3}-\ref{2.4}).

To prove the opposite inequality
we need some additional definitions:
\begin{defi} .
\begin{enumerate}\label{def2}
 \item
Denote $Z_i^j\in A_i$ to be  the closest point in $A_i$  to $A_j$. (see Definition~\ref{def1}-(iii)).
\item  For $i,j\in V$ let
 $\bar{o}_{i,j}=(o^{(1)}_{i,j}\ldots o_{i,j}^{(n)})$,   a choice of an optimal orbit
 realizing (\ref{oo}) in Definition~\ref{def1}-(v) (note that there can be more than one such orbit, but we choose only one).
Let $|\bar{o}_{i,j}|$ be the cardinality of $\bar{o}_{i,j}$.

For any $l\in V$, denote $E^+_l$ the set if all outgoing  edges from $l$, that is, $E \in E^+_l$ iff, for some $i,j\in V$, $l\in \bar{o}_{i,j}=(o^{(1)}_{i,j}\ldots o_{i,j}^{(n)})$, $l\not= o_{i,j}^{(n)}$.
    \par
    Likewise, denote $E^-_l$ the set if all incoming edges to $l$, that is  $E\in E^-_l$ iff $l\in \bar{o}_{i,j}=(o^{(1)}_{i,j}\ldots o_{i,j}^{(n)})$, $l\not= o_{i,j}^{(1)}$.
\item \label{lambdaE}  For each $i,j\in V$, let
$$E_{\bar{o}_{i,j}}:= \{ E; E=E_{o_{i,j}^{(k)}, o_{i,j}^{(k+1)}} ; \ 1\leq k\leq |\bar{o}_{i,j}|-1\}\ . $$
where $\bar{o}_{i,j}$ is the above choice of optimal orbit. Let
\be\label{oiwai} \Lambda_E^*:= \sum_{\{ i,j\ ; E\in{E_{\bar{o}_{i,j}}}\}}\lambda^*_{i,j} \ , \ee
  see (\ref{mmk}) for $\lambda^*_{i,j}$. This is {\em the total flux} traversing $E$ due to the optimal transport plan.
  \par
  Note that
  \be\label{kirk}\sum_{E\in E_l^+}\Lambda_E^* -\sum_{E\in E_l^-}\Lambda_E^* =\bar{\nu}_l\ee
  for any $l\in V$. Recall (Definition \ref{def1} (i,ii)) that $\bar{\nu}_l>0$ if $l\in V_+$, $\bar{\nu}_l<0$ if $l\in V_-$ and $\bar{\nu}_l=0$ if $l\in V-\bar{V}$.
  \par
  Note also that the flux due to optimal plan is uni-directional, i.e  \\ $\Lambda^*_E\cdot\Lambda^*_{-E}=0$ for any edge $E$ (here $-E$ represents the same edge in the opposite orientation).
\item\label{nuplus}
  For $k\in V$
  \be\label{sumE++}\hat{\nu}^+_k:= \sum_{E_{k,i}\in  E^+_k}\Lambda^*_{E_{k,i}}\delta_{{Z_k^i}} \ . \ee
  Here $\delta_x$ is the Dirac delta function at $x$. In particular, $\hat{\nu}^+_k$ is supported in $A_k$ (see Definition \ref{def2} (1)), and
  \be\label{sumE+} |\hat{\nu}^+_k|=\sum_{E\in E^+_k}\Lambda^*_E \ . \ee
\item \label{hatmu}For $i,j\in V$, let $B_r(Z^j_i)$ be the ball of radius $r$ centered at $Z_i^j\in A_i$. Given $\eps>0$ let $r_i^{j,\eps}>0$ be the radius of the ball  such that $\mu\left(A_j\cap B_{r_i^{j,\eps}}(Z_i^j)\right)=\eps$. See Figure 3. \par
 \begin{figure}
\centering
\includegraphics[height=10.cm, width=14.cm]{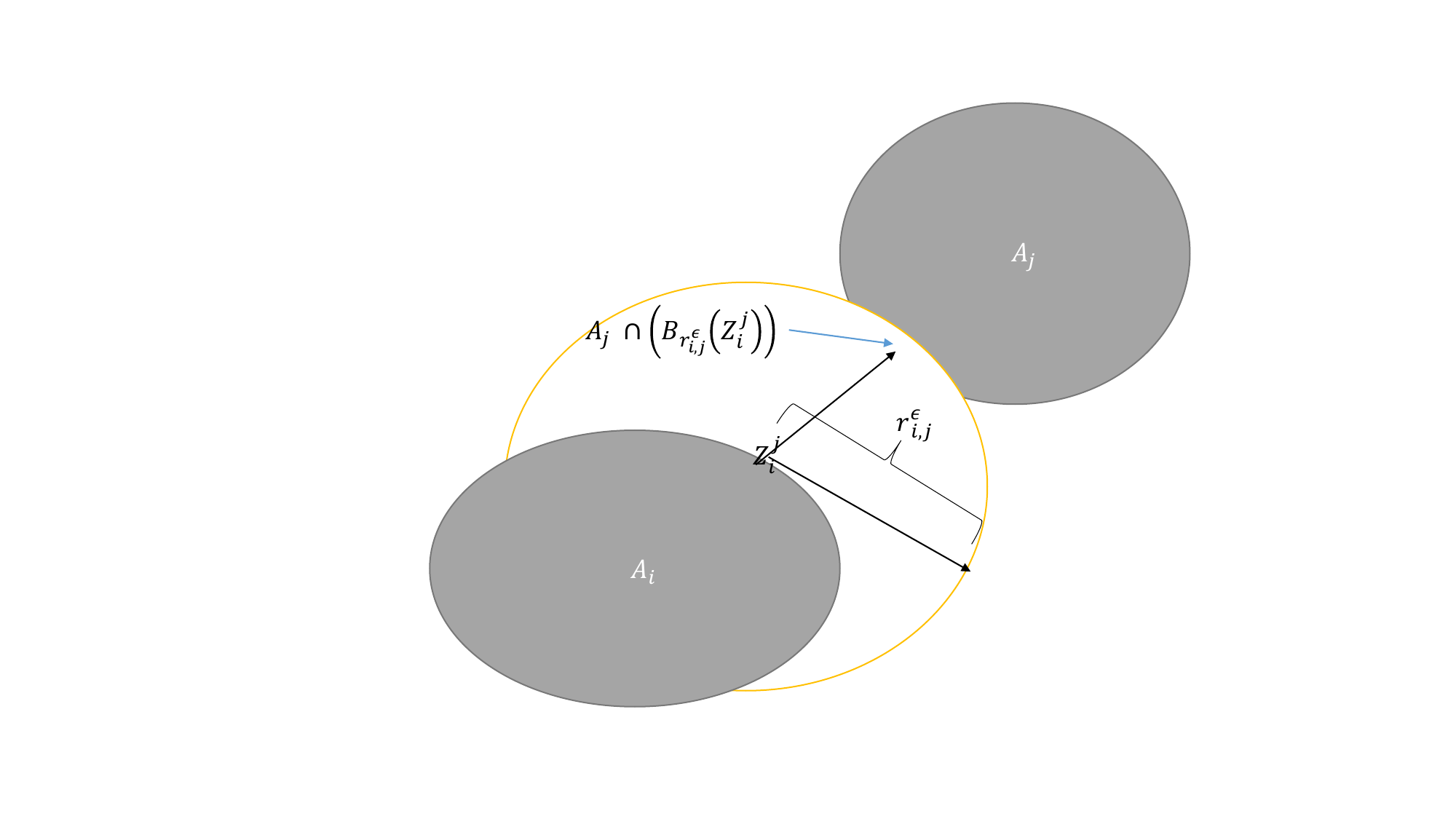}\\
  \caption{}\label{3}
\end{figure}
    Let $\hat{\mu}_{i,j}^\eps$ be the restriction of the measure $\mu$ to the set $A_j\cap B_{r_i^{j,\eps}}(Z_i^j)$  defined above.

\item\label{nu-} Let
\be\label{nu-eq} \hat{\nu}_k^-(\eps):=\sum_{E=E_{l,k}\in E^-_k} \hat{\mu}^{\eps\Lambda_E^*}_{l,k} \ .  \ee
In particular, $\hat{\nu}^-_k$ is supported in $A_k$  and
  \be\label{sumE-} |\hat{\nu}^-_k(\eps)|=\eps\sum_{E\in E^-_k}\Lambda^*_E \ . \ee


\item\label{nufinal} $\hat{\nu}_+:= \sum_{k\in V}\hat{\nu}_k^+$ \ \ ; \ \ $\hat{\nu}_-(\eps):= \sum_{k\in V}\hat{\nu}_k^-(\eps)$ \ \ ; \ \
$\hat{\nu}(\eps):= \eps\hat{\nu}_+-\hat{\nu}_-(\eps)$.

\end{enumerate}
\end{defi}
Note that $\hat{\nu}(\eps)\in{\cal M}_0$, i.e $\eps|\hat{\nu}_+|=|\hat{\nu}_-(\eps)|$. In fact, we obtain from (\ref{kirk}, \ref{sumE+}, \ref{sumE-}) that for each $k\in V$
\be\label{ballancek} \eps|\hat{\nu}^+_k|-|\hat{\nu}^-_k(\eps)|=\eps\bar{\nu}_k \ ,  \ee
and $\sum_{k\in V}\bar{\nu}_k=0$ (Definition \ref{def1}-(i)).
\par

Using the above we find
form the  metric property of $W_p$ and the triangle inequality
\be\label{chainineq}  W_p(\mu+\eps\nu,\mu)\leq  W_p(\mu+\eps\nu,\mu+ \hat{\nu}(\eps))+ W_p(\mu, \mu+ \hat{\nu}(\eps)) \ .  \ee
Let  $\mu_k$ be the restriction of $\mu$ to $A_k$, $\nu_k$  the restriction of $\nu$ to $A_k$ and $\hat{\nu}_k(\eps)=\eps\hat{\nu}^+_k-\hat{\nu}_k^-(\eps)$ . By (\ref{ballancek}) (recall $\bar\nu_k:=|\nu_k|$), $W_p(\mu_k+\eps\nu_k,\mu_k+ \hat{\nu}_k(\eps))$ is defined on each component.  We can use
the definition of Wasserstein metric to obtain
$$  W^p_p(\mu+\eps\nu,\mu+ \hat{\nu}(\eps))\leq \sum_{k\in V} W^p_p(\mu_k+\eps\nu_k,\mu_k+ \hat{\nu}_k(\eps)) \ . $$

   Now Theorem~\ref{firstT} applies  to each of the components of this sum. By the assumption of the Theorem we obtain
  $$W^p_p(\mu_k+\eps\nu^k_+, \mu_k+\eps\hat{\nu}^+_k)=O(\eps^{pq})=o(\eps) \  $$
  where $q>1/p$ by its definition. Thus, the first term on the right of (\ref{chainineq}) is controlled by $o(\eps^{1/p})$.
  \par
 To complete the proof we need to estimate the second term.
\begin{proposition}\label{prop1}
$$ W^p_p(\mu, \mu+\hat{\nu}(\eps))\leq \eps \|\nu\|_\mu + o(\eps) \ . $$
\end{proposition}

For the proof  we construct a transport plan $\pi$ from $\mu$ to $\mu+\hat{\nu}(\eps)$. To illustrate this construction by a particular example see the directed tree in Figure 2. A detailed description of the plan is given below.
\par
For any positive measure $\sigma\in {\cal M}_+(\Omega)$ and $x\in\Omega$ define
 \\ $\delta_x\otimes\sigma\in {\cal M}_+(\Omega\times\Omega)$ by its action on $\phi\in C(\Omega\times\Omega)$:
 $$ <  \delta_x\otimes\sigma,\phi>:= \int_\Omega\phi(x,y)d\sigma(y) \ . $$

Let \be\label{muminus}\mu_-^\eps:= \mu -\hat{\nu}_-(\eps)\ee
and $\pi_{\mu_-^\eps}$ be the diagonal lift of $\mu_-^\eps$ to ${\cal M}_+(\Omega\times\Omega)$, that is,
$$< \pi_{\mu_-^\eps},\phi>:=\int_{\Omega}\phi(x,x)d\mu_-^\eps(x) \ . $$
Let now
 \be\label{pieps}\pi^\eps:= \pi_{\mu_-^\eps}+ \sum_{l\in V}\sum_{k\in V}\delta_{Z_l^k} \otimes\hat{\mu}_{l,k}^{\eps\Lambda^*_{E_{l,k}}} . \ee
Note that some terms in the double sum above my be zero. This is the case if the edge $E_{l,k}$ does  not transverse  an orbit of the optimal transport plan, i.e $\Lambda^*_{E_{l,k}}=0$ (hence $\delta_{Z_l^k} \otimes\hat{\mu}_{l,k}^{\eps\Lambda^*_{E_{l,k}}} =0$).
\par
 Next, observe that $\pi^\eps\in \Pi(\mu+\hat{\nu}(\eps), \mu)$ (c.f (\ref{Pidef})). In fact, from  (\ref{muminus}) and (\ref{pieps}), for any $\phi=1(y)\psi(x)$
\begin{multline}<\pi^\eps,\phi>=\int_\Omega\psi(x) d \mu_-^\eps(x)+\eps\sum_{l\in V}\sum_{k\in V} \psi(Z_k^l)\Lambda^*_{E(k,l)}=\\
 \int_\Omega\psi(x) d \mu(x)-\int_\Omega\psi(x)d\hat{\nu}_-(\eps)(x)+\eps\int_\Omega \psi(x)d\hat{\nu}_+(x) =<\mu+\hat{\nu}(\eps),\psi> \  \end{multline}
 where we used $\hat{\nu}_+:= \sum_{k\in V}\hat{\nu}_k^+$ and (\ref{sumE++}).
 \par
  Setting now $\phi=1(x)\psi(y)$
\begin{multline}<\pi^\eps,\phi>=\int_\Omega\psi(y) d \mu_-^\eps(y)+\sum_{l\in V}\sum_{k\in V} \int_\Omega\psi(y)d\hat{\mu}_{l,k}^{\eps\Lambda^*_{E_{l,k}}} (y)=
 <\mu,\psi>  \  \end{multline}
 where we used $\hat{\nu}_-(\eps):= \sum_{k\in V}\hat{\nu}_k^-(\eps)$ and (\ref{muminus}, \ref{nu-eq}).

It then follows from (\ref{pieps}) that
\be\label{Wpleq} W_p^p(\mu, \mu+\hat{\nu}(\eps))\leq\int_\Omega\int_\Omega |x-y|^pd\pi_\eps=
 \sum_{l\in V}\sum_{E=E_{l,k}}\int_\Omega\int_\Omega |Z_l^k-y|^p \hat{\mu}_{l,k}^{\eps\Lambda^*_E}(dy)
 \ee
 From Definition~\ref{def2}-1,5 and Definition \ref{def1}-iii, we obtain
 $\int_\Omega |Z_l^k-y|^p \hat{\mu}_{l,k}^{\eps\Lambda^*_{E_{l,k}}}(dy)= \eps|E|_{l,k}\Lambda^*_E +o(\eps)$, so Definition \ref{def2}
 -6,7,  together with (\ref{oiwai}) imply
 \be\label{sof} \int_\Omega\int_\Omega |Z_l^k-y|^p \hat{\mu}_{l,k}^{\eps\lambda^*_E}(dy)= \eps|E|_{l,k}\sum_{(i,j); E_{l,k}\in \bar{e}_{i,j}}\lambda^*_{i,j}+o(\eps)\ee
 and  (\ref{Wpleq}, \ref{sof}, \ref{oo}) imply
 $$ W_p^p(\mu, \mu+\hat{\nu}(\eps))\leq \eps\sum_{i,j\in V\times V}\lambda_{i,j}^*d(i,j) +o(\eps)= \eps\|\nu\|_\nu+ o(\eps) \ . $$
\\
$\Box$

\newpage
\begin{center}References\end{center}
\begin{description}
\item{[AGS]} Ambrosio, L, Gigli, N and Sava\'{r}e, G.:{\it Gradient Flows in metric spaces and in the space of probability measures}, Lecture Notes in Mathematics, Birkhauser, 2005
\item{[AG1]}  Ambrosio, L and  Gigli, N: {\it A User's Guide to Optimal Transport} in Modelling and Optimisation
of Flows on Networks
Cetraro, Italy 2009,
B.Piccoli and M. Rascle Ed., Springer
\item{[BB]}  Benamou, J.D,  Brenier, Y.: {\it  A computational fluid mechanics solution to the Monge-Kantorovich mass transfer problem},  Numer. Math., 84, 375-393, 2000
\item{[K]} Kantorovich, L. V. {\it On the translocation of masses.} Dokl. Akad. Nauk. USSR 37 (1942), 199, 201.
English translation in J. Math. Sci. 133, no.4 (2006), 1381, 1382.
\item{[KR]} Kantorovich, L. V., and Rubinshtein, G. S. {\it On a space of totally additive functions.} Vestn.
Leningrad. Univ. 13, 7 (1958), 52,59.
\item{[J]} J\"{a}rvenp\"{a}\"{a}, E.,  J\"{a}rvenp\"{a}\"{a}, M., K\"{a}enm\"{a}ki, A.,
 Rajala, T.,   Rogovin, S and   Suomala, V.: {\it Packing dimension and Ahlfors regularity of porous sets in metric spaces}, Mathematische Zeitschrift
September 2010, Volume 266, Issue 1, pp 83-105
\item{[L]} S. Lisini, {\it Characterization of absolutely continuous curves in Wasserstein spaces},  Calc. Var. Partial Differential Equations, 28, (2007),  85-120
\item{[M]} McCann, R, J.:{\it  A convexity principle for interacting gases}, Adv. Math. 128 (1997), no. 1, 153-179.
\item{[O]} Otto, F. {\it The geometry of dissipative evolution equations: the porous medium equation.} Comm.
Partial Differential Equations 26, 1-2 (2001), 101,174.
\item{[R]}   Rachev, S.T, R\"{u}schendorf, L.R: \ {\it Mass
Transportation Problems}, Vol 1, Springer, 1998
\item{[T]} Tanaka, H. {\it An inequality for a functional of probability distributions and its application to Kac's
one-dimensional model of a Maxwellian gas.} Z. Wahrscheinlichkeitstheorie und Verw. Gebiete 27
(1973), 47, 52.
\item{[Va]} Vasershtein, L. N. {\it Markov processes over denumerable products of spaces describing large system
of automata}. Problemy Peredatci Informacii 5, 3 (1969), 64, 72.
\item{[V]} Villani, C: \ {\it Topics in Optimal Transportation},
Graduate Studies in Mathematics, Vol. 58, AMS
\item{[V1]} Villani, C: \ {\it Optimal Transport; Old and new}, Springer, 2009
\item{[Ve]} Vershik, A. M. {\it The Kantorovich metric: the initial history and little-known applications.} Zap.
Nauchn. Sem. S.-Peterburg. Otdel. Mat. Inst. Steklov. (POMI) 312, Teor. Predst. Din. Sist. Komb. i
Algoritm. Metody. 11 (2004), 69,85, 311.
\item{[W1]} Wolansky, G:{\it  Limit theorems for optimal mass transportation}, , Calc. Var. Partial Differential Equations 42 (2011), no. 3-4,
 487-516

\end{description}

\end{document}